\documentclass[english,british]{article}
\usepackage[T1]{fontenc}
\usepackage[latin9]{inputenc}
\usepackage{geometry}
\usepackage{color}
\usepackage{amsmath}
\usepackage{amssymb}

\makeatletter
\usepackage{amsthm}

\numberwithin{equation}{section}
\numberwithin{figure}{section}
\theoremstyle{plain}
\newtheorem{thm}{\protect\theoremname}[section]
\theoremstyle{plain}
\newtheorem{cor}[thm]{\protect\corollaryname}
\newtheorem{lemma}[thm]{\protect\lemmaname}
\newtheorem{remark}[thm]{Remark}
\theoremstyle{definition}
\newtheorem{example}[thm]{\protect\examplename}

\usepackage{babel}
\addto\captionsbritish{\renewcommand{\corollaryname}{Corollary}}
\addto\captionsbritish{\renewcommand{\lemmaname}{Corollary}}
 \addto\captionsbritish{\renewcommand{\examplename}{Example}}
 \addto\captionsbritish{\renewcommand{\theoremname}{Theorem}}
 \addto\captionsenglish{\renewcommand{\corollaryname}{Corollary}}
\addto\captionsenglish{\renewcommand{\lemmaname}{Lemma}}
 \addto\captionsenglish{\renewcommand{\examplename}{Example}}
 \addto\captionsenglish{\renewcommand{\theoremname}{Theorem}}
 \providecommand{\corollaryname}{Corollary}
\providecommand{\lemmaname}{Corollary}
 \providecommand{\examplename}{Example}
\providecommand{\theoremname}{Theorem}

\newcommand{\R}{\mathbb{R}}

\newcommand{\be}{\begin{equation}}
\newcommand{\ee}{\end{equation}}
\newcommand{\beq}{\begin{eqnarray}}
\newcommand{\eeq}{\end{eqnarray}}
\newcommand{\w}{\omega}
\newcommand{\s}{\epsilon}
\usepackage{babel}

\makeatother

\usepackage{babel}
\begin{document}

\title{On a class of systems of hyperbolic equations describing pseudo-spherical or spherical surfaces }

\author{Filipe Kelmer$^{1}$\quad{and}\quad Keti Tenenblat$^{2}$}
\date{}
\maketitle
\footnotetext[1]{Department of Mathematics, Universidade de Bras\'\i lia,Brazil,  e-mail: kelmer.a.f@gmail.com. Partially supported by CNPq grant 132908/2015-8 and CAPES/Brazil-Finance Code 001.}
\footnotetext[2]{Department of Mathematics, Universidade
de Brasilia, Brazil, e-mail: K.Tenenblat@mat.unb.br.
Partially supported by CNPq 
grant 312462/2014-0 and CAPES/Brazil-Finance Code 001.} 

\begin{abstract}
We consider systems of  partial differential equations of the form
\beq\nonumber
\left\{
\begin{array}{l}
u_{xt}=F\left(u,u_x,v,v_x\right),\\
v_{xt}=G\left(u,u_x,v,v_x\right),
\end{array}
\right. 
\eeq
describing pseudospherical (\textbf{pss}) or spherical surfaces (\textbf{ss}), meaning that, their generic solutions $u(x,t)\, v(x,t)$ provide metrics, with coordinates $(x,t)$, on open subsets of the plane, with constant curvature $K=-1$ or $K=1$. These systems can be described as the integrability conditions of $\mathfrak{g}$-valued linear problems, with $\mathfrak{g}=\mathfrak{sl}(2,\R)$ or $\mathfrak{g}=\mathfrak{su}(2)$, when $K=-1$, $K=1$, respectively. We obtain characterization and also classification results. Applications of the theory provide new examples and new families of systems of differential equations, which contain generalizations of a Pohlmeyer-Lund-Regge type system  and of the Konno-Oono coupled dispersionless system.
\end{abstract}
\noindent 2010 \foreignlanguage{english}{\textit{Mathematics Subject
Classification}: 35L51, 53B20, 58J60}

%35L51: Second-order hyperbolic systems
%58J60: Relations of PDEs with special manifold structures (Riemannian, Finsler, etc.)
%53B20: Local Riemannian geometry

\selectlanguage{english}%

%\selectlanguage{british}%

\section{Introduction}

Systems of partial differential equations describing pseudospherical  or spherical surfaces (\textbf{ss}) are characterized by the fact that their generic solutions provide metrics on  non empty open subsets of $\mathbb{R}^{2}$, with Gaussian curvature $K=-1$ or $K=1$, respectively. This concept was first introduced, in 2002, by Q. Ding e K. Tenenblat in \cite{ding} as a generalization of the notion of differential equations describing \textbf{pss} given in 1986 by S. S. Chern and K. Tenenblat in \cite{chern}.

 The definition given by Chern and Tenenblat was inspired by the Sasaki observation, made in 1979 \cite{sasaki}, that a class of nonlinear differential equations, such as KdV, MKdV and SG which can be solved by the AKNS $2\times 2$ inverse scattering method \cite{akns}, was related to \textbf{pss}. Nowadays it is known that the class of differential equation describing \textbf{pss} is, in fact, larger than the AKNS class.
 
Besides the notion of differential equations describing \textbf{pss},  Chern and Tenenblat introduced a systematic procedure of characterizing and classifying such equations. This procedure has been used for several classes of partial differential equations 
\cite{castrosilva2015, catalanoferraioli2020, catalanoferraioli2016, catalanoferraioli2014, chern, gomesneto2010,  jorge1987,  kamran1995, rabelo1989, rabelo1990, rabelo1992, reyes1998}.

Equations describing \textbf{pss} (resp. \textbf{ss}) can be seen as the compatibility condition of an associated $\mathfrak{su}(2)$-valued (resp. $\mathfrak{sl}\left(2,\mathbb{R}\right)$-valued) linear problem, also referred to as a zero curvature representation. This characterization implies that those equations may present other properties such as B{\"a}cklund transformations \cite{chern,beals1991,Reyes-backlund}, non-local symmetries \cite{Ray7,Ray8},   and an infinite number of conservation laws \cite{cavalcante1988}. Besides that, they are natural candidates for being solved by Inverse Scattering Method \cite{akns,beals1991}.

Another remarkable property of partial differential equations describing \textbf{pss} (resp. \textbf{ss}) is the theoretical existence of local transformations between generic solutions of those equations. This is due to a basic geometric fact that given two points of two Riemannian manifolds, with the same dimension and same constant sectional curvature, there is always an isometry between neighborhoods of those points. Kamran and Tenenblat, in \cite{kamran1995}, explored this property and demonstrated a local existence theorem, assuring that, given any two equations describing \textbf{pss}, then, under a technical assumption, there exists a local smooth application that maps generic solutions of one equation into solutions of the other.

In 2002, Ding and Tenenblat \cite{ding}, besides introducing the notion of systems of partial equations describing \textbf{pss} and \textbf{ss}, they presented characterization results for systems of evolution equations of type
 \beq\nonumber
 \left\{
 \begin{array}{l}
 u_t=F(u,u_x,v,v_x),\\
 v_t=G(u,u_x,v,v_x),
 \end{array}
 \right. 
 \eeq
for $F$ e $G$ smooth functions. In particular, they considered important equations such as nonlinear Schr{\"o}dinger equation, Heisenberg ferromagnetic model and Landau-Lifschitz equations. They also presented a classification result for systems of evolution equations of type
 \beq\nonumber
\left\{
 \begin{array}{l}
 u_t=-v_{xx}+H_{11}(u,v)u_x+H_{12}(u,v)v_x+H_{13}(u,v),\\
 v_t=u_{xx}+H_{21}(u,v)u_x+H_{22}(u,v)v_x+H_{23}(u,v),
 \end{array}
 \right. 
 \eeq
where $F,G,H_{ij}$, $1\leq i\leq 2$, $1\leq j \leq 3$, are smooth functions. Besides these results, very little is known on systems of differential equations describing \textbf{pss} or \textbf{ss}. 

In this paper, we  consider systems of differential equations describing \textbf{pss} or \textbf{ss} of type
 \beq\label{eq:S}
\left\{
 \begin{array}{l}
 \displaystyle u_{xt}=F\left(u,u_x,v,v_x\right),\\
 \displaystyle v_{xt}=G\left(u,u_x,v,v_x\right),
 \end{array}
 \right. 
 \eeq
for $F$ and $G$ smooth functions. We first characterize such systems and then, 
by imposing certain conditions, we obtain some classification results. Moreover, applications of  the theory provide new examples and new families of systems of differential equations, which contain generalizations of a Pohlmeyer-Lund-Regge type system  and  the Konno-Oono coupled dispersionless system.

This paper is organized as follows. In Section \ref{sec:Preliminaries},
we collect some preliminaries on systems of differential equations that describe \textbf{pss} or \textbf{ss} and we provide the linear problem associated to these systems. In Section \ref{section-mainresults}, we start providing several explicit examples,  including a Pohlmeyer-Lund-Regge type of system, and the Konno-Oono  dispersionless system and then we state our main results:  a characterization result (Theorem \ref{Lemma:S10}) and two classification results (Theorem \ref{th:I} and Theorem \ref{th:II}).  In Section \ref{sec:Proofs}, we  prove Theorems \ref{Lemma:S10}, \ref{th:I} and \ref{th:II}. Finally in Section \ref{further_ex}, as an application of the classification results we present four corollaries, which provide 
additional examples and 
new families of systems of differential equations describing \textbf{pss} and \textbf{ss}.

\section{\label{sec:Preliminaries}Preliminaries}

If $\left(M,\, \mathbf{g}\right)$ is a 2-dimensional Riemannian manifold and $\left\{ \omega_{1},\omega_{2}\right\} $ is a co-frame, dual to an orthonormal frame $\left\{ e_{1},e_{2}\right\} $, then the metric is given by  $\mathbf{g}=\omega_{1}^{2}+\omega_{2}^{2}$ and $\omega_{i}$ satisfy the structure equations: $d\omega_{1}=\omega_{3}\wedge\omega_{2}$ and $d\omega_{2}=\omega_{1}\wedge\omega_{3}$, where $\omega_{3}$ denotes the connection form defined as $\omega_{3}(e_{i})=d\omega_{i}(e_{1},e_{2})$. The Gaussian curvature of $M$ is the function $K$ such that $d\omega_{3}=-K\omega_{1}\wedge\omega_{2}$.

A system of partial differential equations $\mathcal{S}$, for scalar functions $u\left(x,t\right)$ and $v\left(x,t\right)$ \emph{describes
pseudospherical surfaces}\textbf{(pss)}\emph{, or spherical surfaces
}\textbf{(ss)} if it is equivalent to the structure equations (see \cite{keti})
of a surface with Gaussian curvature $K=-\delta$, with $\delta=1$
or $\delta=-1$, respectively, i.e., 
\begin{equation}\label{eq:SE}
\begin{array}{l}
d\omega_{1}=\omega_{3}\wedge\omega_{2},\quad d\omega_{2}=\omega_{1}\wedge\omega_{3},\quad d\omega_{3}=\delta\omega_{1}\wedge\omega_{2},\end{array}
\end{equation}
where $\left\{ \omega_{1},\omega_{2},\omega_{3}\right\} $ are $1$-forms
\begin{equation}
\begin{array}{l}
\omega_{1}=f_{11}dx+f_{12}dt,\quad\omega_{2}=f_{21}dx+f_{22}dt,\quad\omega_{3}=f_{31}dx+f_{32}dt,\end{array}\label{eq:forms}
\end{equation}

\noindent such that $\omega_{1}\wedge\omega_{2}\neq0$, i.e., 
\begin{equation}
f_{11}f_{22}-f_{12}f_{21}\neq0,\label{eq:nondeg_cond}
\end{equation}
and $f_{ij}$ are functions of $x$, $t$, $u(x,t)$, $v(x,t)$ and it's derivatives with respect to $x$ and $t$.

Locally, considering the basis $\left\{ dx,dt\right\} $, the structure equations (\ref{eq:SE}) are equivalent to the following system of equations
\beq\label{eq:SELocal}
\left\{
\begin{array}{l}
-f_{11,t}+f_{12,x}=f_{31}f_{22}-f_{32}f_{21},\\
-f_{21,t}+f_{22,x}=f_{11}f_{32}-f_{12}f_{31},\\
-f_{31,t}+f_{32,x}=\delta(f_{11}f_{22}-f_{12}f_{21}),
\end{array}
\right.
\eeq
and the metric is
\be\nonumber
ds^2=(f_{11}^2+f_{21}^2)dx^2+2(f_{11}f_{12}+f_{21}f_{22})dxdt+(f_{12}^2+f_{22}^2)dt^2.
\ee
Notice that, according to the definition, given a solution $u,v$ of a system $\mathcal{S}$ describing \textbf{pss} (or \textbf{ss}) with associated 1-forms $\omega_{1},\omega_{2}$ and $\omega_{3}$, we consider an open connected set $U\subset\mathbb{R}^{2}$, contained in the domain of $u,v$, where $\omega_{1}\wedge\omega_{2}$ is everywhere nonzero on $U$. Such an open set $U$ exists for generic solutions $z$. Then, $\mathbf{g}=\omega_{1}^{2}+\omega_{2}^{2}$ defines a Riemannian metric, on $U$, with Gaussian curvature $K=-\delta$. It is in this sense that one can say that a system describes, \textbf{pss} (\textbf{ss}, resp.). This is an intrinsic  geometric property of a  Riemannian metric 
(not immersed in an ambient space).

A classical example is the nonlinear Schr{\"o}dinger equation,
\beq\label{eq:NLSE-}
\left\{
\begin{array}{l}
u_t+v_{xx}-2(u^2+v^2)v=0,\\
-v_t+u_{xx}-2(u^2+v^2)u=0,
\end{array}
\right. 
\eeq
which is a system describing \textbf{pss},  with associated functions
\beq\label{eq:NLSE-fij}
\begin{array}{lll}
f_{11}=2u,& f_{21}=-2v,& f_{31}=2\eta\\
f_{12}=-4\eta u-2v_{x}, & f_{22}=4\eta v-2u_{x}, & f_{32}=-4\eta^2-2(u^2+v^2),
\end{array}
\eeq
where $\eta\in\R$ is a parameter. Indeed, for the functions  $f_{ij}$ above, the structure equations (\ref{eq:SELocal}), with $\delta=1$, is satisfied modulo (\ref{eq:NLSE-}). In this example, for every generic solution $u(x,t),v(x,t)$, of (\ref{eq:NLSE-}), such that condition (\ref{eq:nondeg_cond})
holds, i.e.,  $-2(u^2+v^2)_x\neq 0$, there exists a Riemannian metric $\mathbf {g}$, with constant Gaussian curvature $K=-1$, whose coefficients $g_{ij}$ are  given by 
\be
\begin{array}{l}
g_{11}=4(u^2+v^2),\\
g_{12}=g_{21}=-16\eta(u^2+v^2)-8(uv_x-vu_x),\\
g_{22}=16\eta^2(u^2+v^2)+4(u_x^2+v_x^2)+16\eta(uv_x-vu_x).
\end{array}
\ee
In this case, the parameter $\eta$, implies that there exists a one-parameter family of metrics associated to every generic solution. Moreover, the presence of this parameter may be related to the existence of an infinite number of conservation laws, B{\"a}cklund transformations, and to the possibility of solving the system by applying the inverse scattering method.

Another classical example of a system describing \textbf{ss}, see \cite{ding}, is the nonlinear Schr{\"o}dinger system ($NLS^{+}$),
\be\label{eq:NLSE+}
\left\{ \begin{array}{l}
u_{t}+v_{xx}+2\left(u^{2}+v^{2}\right)u=0,\vspace{4pt}\\
-v_{t}+u_{xx}+2\left(u^{2}+v^{2}\right)v=0,
\end{array}\right.
\ee
with associated functions
\be\nonumber
\begin{array}{lll}
f_{11}=2v,& f_{21}=2\eta, & f_{31}=-2u,\\
f_{12}=-4\eta v+2u_{x}, & f_{22}=-4\eta^{2}+2\left(u^{2}+v^{2}\right),& f_{32}=2\eta u+2v_{x},
\end{array}
\ee
with $\eta\in\mathbb{R}$. Indeed, for the $f_{ij}$ above the structure equations (\ref{eq:SELocal}), with $\delta=-1$, is satisfied modulo (\ref{eq:NLSE+}).

A system of differential equations describing \textbf{pss}, or \textbf{ss}, can be seen as the integrability condition
\be \label{eq:CondIntegrabilidade}
 d\Omega-\Omega\wedge \Omega=0,
\ee
of the linear system \cite{chern},
 \be\label{eq:ProblemaLinear}
 d\Psi=\Omega \Psi,
 \ee
where $\Psi=\left(\Psi_1\\ \Psi_2\right)^T$, with $\Psi_i(x,t)$, $i=1,2,$ and $\Omega$ is either the $\mathfrak{sl}\left(2,\mathbb{R}\right)$-valued
1-form
\[
\Omega=\frac{1}{2}\left(\begin{array}{cc}
\omega_{2} & \omega_{1}-\omega_{3}\\
\omega_{1}+\omega_{3} & -\omega_{2}
\end{array}\right),\qquad\text{when \;}\delta=1,
\]
or the $\mathfrak{su}\left(2\right)$-valued 1-form 
\[
\Omega=\frac{1}{2}\left(\begin{array}{cc}
i\omega_{2} & \omega_{1}+i\omega_{3}\\
-\omega_{1}+i\omega_{3} & -i\omega_{2}
\end{array}\right),\qquad\text{when \;}\delta=-1.
\]
This characterization is due to the fact that the structure equations  (\ref{eq:SE}) are equivalent to (\ref{eq:CondIntegrabilidade}).

Locally, considering $\Omega=A dx+B dt$, the linear problem (\ref{eq:ProblemaLinear}), is given by
\be\label{eq:ProblemaLinearlocal}
\Psi_x=A\Psi,\qquad\Psi_t=B \Psi,
\ee
where
\be\label{eq:ABlocal}
A=\frac{1}{2}\left(\begin{array}{cc}
f_{21}&f_{11}-f_{31}\\
f_{11}+f_{31}&-f_{21}
\end{array}\right),
\quad
B=\frac{1}{2}
\left(\begin{array}{cc}
f_{22}&f_{12}-f_{32}\\
f_{12}+f_{32}&-f_{22}
\end{array}\right),
\ee
when $\delta =1$ or,
\be\label{eq:ABsu2}
A=\frac{1}{2}\left(\begin{array}{cc}
if_{21}&f_{11}+if_{31}\\
-f_{11}+if_{31}&-if_{21}
\end{array}\right),
\quad
B=\frac{1}{2}
\left(\begin{array}{cc}
if_{22}&f_{12}+if_{32}\\
-f_{12}+if_{32}&-if_{22}
\end{array}\right),
\ee
when $\delta=-1$. Moreover, the integrability condition, $\displaystyle \Psi_{xt}=\Psi_{tx}$, is given by
\be\label{eq:compatibilidade2x2}
 A_t-B_x+AB-BA=0,
\ee
which is also known as $\mathfrak{sl}\left(2,\mathbb{R}\right)$ (or $\mathfrak{su}\left(2\right)$) \emph{zero-curvature representation}.

Alternatively, a system describing \textbf{pss} (resp. \textbf{ss}) is the  integrability condition of another type of a linear problem \cite{castrosilva2015},
\be\label{eq:ProblemaLinear3x3}
d\hat{\Psi}=\hat{\Omega} \hat{\Psi},
\ee
for $\hat{\Psi}=\left(\hat{\Psi}_1,\hat{\Psi}_2, \hat{\Psi}_3\right)^T$ and
\be\nonumber
\hat{\Omega}=
\left(\begin{array}{ccc}
0 & \w_1 & \w_2 \\
\delta \w_1 & 0 & \w_3 \\
\delta \w_2 & -\w_3 & 0  
\end{array}\right),
\ee
where $\delta=1$ ($-1$ resp.). In local coordinates, $\hat{\Omega}=\hat{A} dx+\hat{B} dt$, the linear problem (\ref{eq:ProblemaLinear3x3}) is equivalent to
\be\label{eq:ProblemaLinearlocal3x3}
\hat{\Psi}_x= \hat{A}\hat{\Psi},\qquad\hat{\Psi}_t =\hat{B} \hat{\Psi},
\ee
where
\be\label{eq:ABlocal3x3}
\hat{A}=\left(\begin{array}{ccc}
0             &  f_{11}&   f_{21}\\
\delta f_{11} & 0      &   f_{31}\\
\delta f_{21} & -f_{31} &    0
\end{array}\right),
\quad
\hat{B}=\left(\begin{array}{ccc}
0             &  f_{12}&   f_{22}\\
\delta f_{12} & 0      &   f_{32}\\
\delta f_{22} & -f_{32} &    0
\end{array}\right),
\ee
with $\delta=1$ (resp. $\delta=-1$),  where $\hat{A}$ and $\hat{B}$ belong to $\mathfrak{so}(2,1)$   (resp. $\mathfrak{so}(3)$).  
These matrices are related to $A$ and $B$ by the following Lie algebra isomorphisms. Since $SL(2,\R)$ is locally isomorphic to $SO^+(2,1)$, we have the isomorphism 
$\mathfrak{so}(2,1)\simeq \mathfrak{sl}(2,\R)$. Similarly, the local isomorphism of $SO(3)$ with $SU(2)$ provides the Lie algebra isomorphism 
$\mathfrak{so}(3)\simeq \mathfrak{su}(2)$

As an example, consider the nonlinear Schr{\"o}dinger equation (\ref{eq:NLSE-}).  Since it describes a \textbf{pss},   it follows that it  is the integrability condition (\ref{eq:compatibilidade2x2}) of the linear problem (\ref{eq:ProblemaLinearlocal}), with $A$ and $B$ defined by (\ref{eq:ABlocal}),    where the functions $f_{ij}$ are given by (\ref{eq:NLSE-fij}). I.e., considering the linear problem
\beq\nonumber
\Psi_x&=&\left(\begin{array}{cc}
-v & u-\eta\\
u+\eta & v
\end{array}\right)\Psi,\\\nonumber
\Psi_t&=&
\left(\begin{array}{cc}
2\eta v-u_{x}&-2\eta u-v_{x}+2\eta^2+u^2+v^2\\
-2\eta u-v_{x}-2\eta^2-u^2-v^2&-2\eta v+u_{x}
\end{array}\right)
\Psi,
\eeq
the integrability condition, $\Psi_{xt}=\Psi_{tx}$, is satisfied if, and only if, the pair $u,\,v$ is a solution of the  nonlinear Schr{\"o}dinger equation (\ref{eq:NLSE-}). Alternatively, it is also the integrability condition of the $3\times 3$ linear problem (\ref{eq:ProblemaLinearlocal3x3}), where $\hat{A}$ and $\hat{B}$ are given by (\ref{eq:ABlocal3x3}),  with $\delta=1$
 and $f_{ij}$ defined by (\ref{eq:NLSE-fij}). 

\section{Some examples and main results \label{section-mainresults}}

In this section, we first illustrate the class of systems of partial differential equations of the form (\ref{eq:S}), studied in this paper,
 with some examples which  describe \textbf{pss} or \textbf{ss}. Then, we present a characterization result given by Theorem \ref{Lemma:S10}, 
 that gives necessary and sufficient conditions for a system  (\ref{eq:S}) 
to  describe \textbf{pss} or \textbf{ss} and  we  state our main classification results given by  Theorems \ref{th:I}, \ref{th:II} and \ref{th:III}. The proof of these theorems is postponed to Section \ref{sec:Proofs} and  further examples provided by our main results are given in Section \ref{further_ex}.\\

\subsection{Examples \label{first_ex}}

In this subsection, we provide several examples of systems of hyperbolic equations of type (\ref{eq:S}) that describe \textbf{pss} or \textbf{ss}.
We include a Pholmeyer-Lund-Regge type system \cite{adler2000},  a coupled integrable dispersionless equations studied  by Konno-Oono \cite{konno-oono} that are already known in the literature and some new examples. Further families of examples will be presented in Section \ref{further_ex}.

\medskip{}

\begin{example}\label{exemploPLR}
The following system of differential equations for $u(x,t)$ and $v(x,t)$ 
\beq\label{eq:mPLR1}
\left\lbrace \begin{array}{l}
u_{xt}=2 u v u_{x}-u,\\
v_{xt}=-2 u v v_{x}-v,
\end{array}\right.
\eeq
is a Pholmeyer-Lund-Regge type of system (see \cite{adler2000}). It   describes \textbf{pss} with 
\be\nonumber
\begin{array}{lll}
f_{11}=\eta(u_x+v_x), & f_{21}=\eta^2, & f_{31}=\eta(u_x-v_x),\\
f_{12}=\frac{1}{\eta}(v-u),& f_{22}=-\frac{1}{\eta^2}-2uv,  & f_{32}=-\frac{1}{\eta}(u+v),
\end{array}
\ee
where $\eta\neq 0$ is a real parameter.
\end{example}

\begin{example}\label{exemplo:konnooonno} (Konno-Oono coupled integrable dispersionless system \cite{konno-oono}) The system 
\beq\label{eq:KOCDS}
\left\lbrace \begin{array}{l}
u_{xt}=-2 v v_{x},\\
v_{xt}=2 v u_{x}, 
\end{array}\right.
\eeq
describes  \textbf{pss},  with
\be\nonumber
\begin{array}{lll}
f_{11}=\frac{2}{\nu}v_x,& f_{21}=\frac{2}{\nu}u_x,&f_{31}=0\\ f_{12}=0,& f_{22}=\nu,&  f_{32}=2 v.
\end{array}
\ee
where $\nu\neq 0$ is a real parameter.
\end{example}

\begin{example} The system of differential equations %\ref{cor:7mprlpseudoesferica}.
\beq\label{eq:7mplrex1}
\left\{
\begin{array}{l}
u_{xt}=(u^2-v^2+c)v_x+u,\\
v_{xt}=(u^2-v^2+c)u_x+v,
\end{array}
\right. 
\eeq
where $c\in\R$,  describes \textbf{pss} with
\be\nonumber
\begin{array}{ll}
f_{11}=-\eta\sqrt{2}u_x ,&\qquad f_{12}=\frac{\sqrt{2}}{\eta}v,\\
f_{21}=\eta^2 ,&\qquad f_{22}=\frac{1}{\eta^2}+u^2-v^2+c,\\
f_{31}=\eta\sqrt{2}v_x,& \qquad f_{32}=-\frac{\sqrt{2}}{\eta}u,
\end{array}
\ee
where $\eta\neq 0$. 
\end{example}

\begin{example} The system of differential equations
\beq\label{eq:7mplrex2}
\left\{
\begin{array}{l}
u_{xt}=(u^2+v^2+c)v_x+u,\\
v_{xt}=-(u^2+v^2+c)u_x+v,
\end{array}
\right. 
\eeq
where $c\in\R$, describes \textbf{ss} with
\be\nonumber
\begin{array}{ll}
f_{11}=-\eta\sqrt{2}v_x ,&\qquad f_{12}=\frac{1}{\eta}\sqrt{2}u,\\
f_{21}=\eta^2 ,&\qquad f_{22}=-\frac{1}{\eta^2}+u^2+v^2+c,\\
f_{31}=-\eta\sqrt{2}u_x,& \qquad f_{32}=-\frac{1}{\eta}\sqrt{2}v,
\end{array}
\ee
com $\eta\neq 0$. 
\end{example}

\begin{example} The system 
\beq\label{eq:exeex1}
\left\lbrace\begin{array}{l}
u_{xt}=(av_x+b)e^u,\\
v_{xt}=-\frac{2}{a}e^uu_x,
\end{array}\right.
\eeq
where $a\neq 0$ and $b\in \R$ are constants, describes \textbf{pss} with associated functions
\beq\nonumber
\begin{array}{lll}
f_{11}=u_x, & f_{21}=\eta, & f_{31}=\eta +av_x+b,\\
f_{12}=0, & f_{22}=-e^u, & f_{32}=-e^u,
\end{array}
\eeq
where $\eta\neq 0$.
\end{example}

\begin{example} The following system of differential equations
\beq\label{eq:exeex2}
\left\{
\begin{array}{l}
u_{xt}=2e^{u-v_x},\\
v_{xt}=u_xe^{u+v_x},
\end{array}
\right.
\eeq
 describes \textbf{pss} with associated functions,
\be\nonumber
\begin{array}{lll}
f_{11}=-\frac{\sqrt{2}}{2}u_x, & f_{21}=\eta, & f_{31}=-\eta\sqrt{2}+e^{-v_x},\\
f_{12}=0,& f_{22}=\sqrt{2}e^{u},& f_{32}=-2e^{u}.
\end{array}
\ee
where $\eta\neq 0$.
\end{example}

\begin{example} The system of differential equations
\beq\label{eq:exeex3}
\left\{
\begin{array}{l}
u_{xt}=-2v_xe^{u},\\
v_{xt}=u_xe^{u},
\end{array}
\right.
\eeq
describes \textbf{ss} with associated functions,
\be\nonumber
\begin{array}{lll}
f_{11}=-\frac{\sqrt{2}}{2}u_x,&f_{21}=\eta, &f_{31}=-\eta\sqrt{2}+v_x\\
f_{12}=0,& f_{22}=-\sqrt{2}e^{u},& f_{32}=2e^{u},
\end{array}
\ee
where $\eta\neq 0$.
\end{example}

\begin{example} The following system 
\be\label{eq:exquatroparametros1}
\left\lbrace\begin{array}{l}
u_{xt}=(au+b)\phi(v_x)+1,\\
v_{xt}=\delta\dfrac{a^2}{\phi'(v_x)}u_x(au+b),
\end{array}\right.
\ee
where $\delta=1$ (resp. $\delta=-1$), $a,b\in \R$ are constants, $a\neq 0$ and $\phi(v_x)$ is a strictly monotone smooth function of $v_x$,  describes \textbf{pss} (resp.\textbf{ss}). In this case, the associated functions are
\be\nonumber
\begin{array}{lll}
f_{11}=\eta a u_x,& f_{21}=\eta^2,& f_{31}=-\eta\phi(v_x),\\
f_{12}=0,& f_{22}=a^2u+ab,& f_{32}=\dfrac{a}{\eta},
\end{array}
\ee
where $\eta\neq 0$ is a real parameter.
\end{example}

\begin{example} The following system
\beq\label{eq:cincoparametrosKODISex1}
\left\{
\begin{array}{l}
u_{xt}=uu_xv_x,\\
v_{xt}=-u(v_x^2+1),
\end{array}
\right.
\eeq
 describes \textbf{pss} with associated functions
\be\nonumber
\begin{array}{ll}
f_{11}=\sigma u_x/{\nu},&\qquad f_{12}=0,\\
f_{21}=u_xv_x/\nu,&\qquad f_{22}=\nu,\\
%\displaystyle{f_{11}=\frac{\sigma u_x}{\nu} },&\qquad f_{12}=0,\\
%\vspace*{.05in}
%\displaystyle{f_{21}=\frac{u_xv_x}{\nu} },&\qquad f_{22}=\nu,\\
f_{31}=0,& \qquad f_{32}=\sigma u,
\end{array}
\ee
where, $\sigma=\pm 1$ and $\nu\in\R\setminus\{ 0\}$.
\end{example}

\subsection{Main theorems}
From now on, we will use the following notation,
\beq
u=z,\quad u_x=z_1,\nonumber\\
v=y,\quad v_x=y_1.\nonumber
\eeq
With this notation, the system (\ref{eq:S}) can be written as
\be\label{eq:S'}
\left\{
\begin{array}{l}
z_{1,t}=F\left(z,z_1,y,y_1\right),\\
y_{1,t}=G\left(z,z_1,y,y_1\right).
\end{array}
\right. 
\ee

We now state our characterization result.

\begin{thm}\label{Lemma:S10} The necessary and sufficient conditions for the system of differential equations (\ref{eq:S'}) to describes \textbf{pss} (resp. \textbf{ss}), with associated functions $f_{ij}=f_{ij}(z,z_{1},y,y_1)$, are
\be\label{eq:lemma1-1}
f_{i1,z}=0,\quad f_{i1,y}=0,\quad f_{i2,z_1}=0,\quad f_{i2,y_1}=0,\quad 1\leq i\leq 3,
\ee
\be\label{eq:lemma1-2}
{\left\vert
\begin{array}{cc}
f_{11,z_1}&f_{11,y_1}\\
f_{21,z_1}&f_{21,y_1}
\end{array}
\right\vert}^2
 +
{\left\vert
\begin{array}{cc}
f_{21,z_1}&f_{21,y_1}\\
f_{31,z_1}&f_{31,y_1}
\end{array}
\right\vert}^2
+
{\left\vert
\begin{array}{cc}
f_{11,z_1}&f_{11,y_1}\\
f_{31,z_1}&f_{31,y_1}
\end{array}
\right\vert}^2\neq 0,
\ee
\be\label{eq:lemma1-3}
-f_{11,z_1}F-f_{11,y_1}G+f_{12,z}z_1+f_{12,y}y_1-f_{31}f_{22}+f_{32}f_{21}=0,
\ee
\be\label{eq:lemma1-4}
-f_{21,z_1}F-f_{21,y_1}G+f_{22,z}z_1+f_{22,y}y_1-f_{11}f_{32}+f_{12}f_{31}=0,
\ee
\be\label{eq:lemma1-5}
-f_{31,z_1}F-f_{31,y_1}G+f_{32,z}z_1+f_{32,y}y_1-\delta f_{11}f_{22}+\delta f_{12}f_{21}=0,
\ee
\be\label{eq:lemma1-6}
f_{11}f_{22}-f_{12}f_{21}\neq 0,
\ee
where $\delta=1$ (resp. $\delta=-1$).
\end{thm}

Our next theorems are  classification results for systems of type \eqref{eq:S'},  which describe \textbf{pss} or \textbf{ss},  obtained under the hypothesis that the associated functions $f_{ij}$ depend on $(z,z_1,y,y_1)$, as in Theorem \ref{Lemma:S10}, and at least one of the functions $f_{31}, f_{21}$ or $f_{11}$ is a constant $\eta\in\R$. Each one of these cases is treated in  Theorems \ref{th:I}, \ref{th:II} and \ref{th:III}. This assumption was inspired by Examples \ref{exemploPLR} and \ref{eq:KOCDS}, where  $f_{21}$ and $f_{31}$, respectively, are constant. 
Moreover, we are interested in systems (\ref{eq:S'}) satisfying a generic condition,
\be\label{eq:irr}
F_{z}^2+F_y^2+G_z^2+G_y^2\neq 0,
\ee
up to a set of measure zero. This condition is not particularly restrictive and it prevents the system (\ref{eq:S'}) to be reduced, up to a change of variables, to a system of evolution equations already considered by Ding and Tenenblat \cite{ding}.

 We notice that  more general classification results may be investigated  by dropping the assumption that one of the functions $f_{i1}$ is constant or assuming that $f_{ij}$ depend explicitly on $x$, $t$ or derivatives of $u$ and $v$  with respect to $t$.

\begin{thm}\label{th:I}
 A system of type (\ref{eq:S'}), satisfying (\ref{eq:irr}), describes \textbf{pss} (resp. \textbf{ss}), with associated functions $f_{ij}(z,z_{1},y,y_1)$, such that $f_{31}=\eta\in\R$ if, and only if, one of the two 
 following cases occur
\begin{description}
\item [{(i)}]
\beq\label{eq:thIAFG}
\left\lbrace\begin{array}{l}
z_{1,t}=\frac{1}{W}\left[  -\delta a(bz_1+ay_1)\phi''(\xi)-\eta (\mu h_{y_1}+\lambda g_{y_1})\phi'(\xi)+\frac{1}{2}(g^2+h^2)_{y_1}\phi (\xi) \right], \\
y_{1,t}=\frac{1}{W}\left[\delta b(bz_1+ay_1)\phi''(\xi)+\eta(\mu h_{z_1}+\lambda g_{z_1})\phi'(\xi)-\frac{1}{2}(g^2+h^2)_{z_1}\phi (\xi) \right],
\end{array}
\right.
\eeq
where  $a,b,\lambda,\mu\in\R$, $a^2+b^2\neq 0$, $\lambda^2 +\mu^2\neq 0$,  $g(z_1,y_1)$, $h(z_1,y_1)$ are smooth functions such that $\mu g -\lambda h=\delta a y_1+\delta b z_1$, with $\delta=1$ (resp. $\delta=-1$), $W:=g_{z_1}h_{y_1}-g_{y_1}h_{z_1}\neq 0$ and $\phi(\xi)$ is a strictly monotone function of $\xi=a y+bz$. 
In this case, the associated functions are 
\be\label{eq:fijS1f31Caso1}
\begin{array}{ll}
f_{11}=g, \qquad& f_{12}=\lambda \phi'(\xi),\\
f_{21}=h, \qquad& f_{22}=\mu \phi'(\xi),\\
f_{31}=\eta, \qquad& f_{32}=\phi(\xi).
\end{array}
\ee
\item [{(ii)}] 
\beq\label{eq:thIBFG}
\left\lbrace\begin{array}{l}
z_{1,t}=-\dfrac{z_1}{\gamma}(\delta p_{zy}-\tau p)+\dfrac{y_1}{\gamma}(-\delta p_{yy}+ \beta p)+\dfrac{\delta\eta }{\gamma^2}(-\beta p_{z}+\tau p_{y}),\\
y_{1,t}=\dfrac{z_1}{\gamma^2}(\delta p_{zz}-\alpha p)+\dfrac{y_1}{\gamma}(\delta p_{zy}-\tau p)+\dfrac{\delta\eta}{\gamma^2}(\tau p_{z}-\alpha p_{y}),
\end{array}\right.
\eeq
where  $\gamma:=a_1b_2-a_2b_1\neq 0$,  with $a_1,b_1,a_2, b_2\in\R$,
$\alpha :=a_1^2+a_2^2$, $\beta:=b_1^2+ b_2^2$, $\tau:=a_1b_1+ a_2b_2$, $\delta=1$ (resp. $\delta=-1$) and $p(z,y)$ is a smooth function such that, $p_z$ e $p_y$ are not proportional. In this case, the associated functions are
\be\label{eq:fijS1f31Caso2}
\begin{array}{ll}
f_{11}=a_{1}z_1+b_{1}y_1,\qquad & f_{12}=\frac{\delta}{\gamma}(b_1 p_{z}-a_1 p_{y}),\\

f_{21}=a_{2}z_1+b_{2}y_1,\qquad & f_{22}=\frac{\delta}{\gamma}(b_2 p_{z}-a_2 p_{y}),\\

f_{31}=\eta, \qquad &f_{32}=p.
\end{array}
\ee
\end{description}
\end{thm}
\begin{thm}\label{th:II}
 A system of type (\ref{eq:S'}), satisfying (\ref{eq:irr}), describes \textbf{pss} (resp. \textbf{ss}) with associated functions $f_{ij}(z,z_{1},y,y_1)$, such that $f_{21}=\eta\in\R$ if, and only if, 
 one of the two 
 following cases occur 
\begin{description}
\item [{(i)}] \beq\label{eq:thIIBFG}
\left\lbrace\begin{array}{l}
z_{1,t}=\frac{1}{W}\left[  -(ab z_1+a^2y_1)\phi''(\xi)+\eta  (\mu h_{y_1}-\delta\lambda g_{y_1})\phi'(\xi)- \frac{1}{2}(h^2-\delta g^2)_{y_1}\phi (\xi) \right],\\
y_{1,t}=\frac{1}{W}\left[ (b^2 z_1+aby_1)\phi''(\xi)-\eta(\mu h_{z_1}-\delta\lambda g_{z_1})\phi'(\xi)+\frac{1}{2}(h^2-\delta g^2)_{z_1}\phi (\xi) \right],
\end{array}\right.
\eeq
where, $\delta=1$ (resp. $-1$), $a,b,\mu,\lambda\in\R$, \;  $g(z_1,y_1), \,h(z_1,y_1)$ and $\phi(\xi)$, $\xi=a y+b z$, are smooth functions, such that $\phi$ is a strictly monotone function of $\xi$ and  
\beq
&a^2+b^2\neq 0,\qquad \mu^2+\lambda^2\neq 0,\qquad W=g_{z_1}h_{y_1}-g_{y_1}h_{z_1}\neq 0,\nonumber\\
&\mu g-\lambda h= ay_1+b z_1,\qquad 
g\neq \lambda \eta \frac{\phi'(\xi)}{\phi(\xi)}.\label{eq:condThII1}
\eeq
In this case, the associated function are
\be\label{eq:fijS1f21Caso1}
\begin{array}{ll}
f_{11}=g,\qquad & f_{12}=\lambda \phi'(\xi),\\
f_{21}=\eta,\qquad &f_{22}=\phi(\xi),\\
f_{31}=h, \qquad &f_{32}=\mu \phi'(\xi).
\end{array}
\ee
\item [{(ii)}]  
\be\label{eq:thIICFG}
\left\lbrace\begin{array}{l}
z_{1,t}=-\dfrac{z_1}{\gamma}(p_{zy}+\tau p)-\dfrac{y_1}{\gamma}(p_{yy}+\beta p)+\dfrac{\eta}{\gamma^2}(\beta p_{z}-\tau p_{y}),\\[8pt]
y_{1,t}=\dfrac{z_1}{\gamma}(p_{zz}+\alpha p)+\dfrac{y_1}{\gamma}(p_{zy}+\tau p)-\dfrac{\eta}{\gamma^2}(\tau p_{z}-\alpha p_{y}),
\end{array}\right.
\ee
where 
\beq
\gamma:=a_1b_3-b_1a_3\neq 0\quad \alpha=a_3^2-\delta a_1^2,\quad  \beta=b_3^2-\delta b_1^2,\quad \tau= a_3b_3-\delta a_1b_1,\label{eq:condThII2}
\eeq
with  $a_1,b_1,a_3,b_3\in\R$, $\delta=1$ (resp. $-1$) and $\, p(z,y)$ is a smooth function, such that $p_z$ and  $p_y$  are not proportional. 
In this case, the associate functions are
\be\label{eq:fijS1f21Caso2}
\begin{array}{ll}
f_{11}=a_{1}z_1+b_{1}y_1,\qquad & f_{12}=\frac{1}{\gamma}(b_1 p_z -a_1 p_y),\\
f_{21}=\eta,\qquad & f_{22}=p,\\
f_{31}=a_{3}z_1+b_{3}y_1,\qquad & f_{32}=\frac{1}{\gamma}(b_3 p_z -a_3 p_y).
\end{array}
\ee
\end{description}
\end{thm}

\begin{thm}\label{th:III} A system of type (\ref{eq:S'}), satisfying (\ref{eq:irr}), describes \textbf{pss} (resp. \textbf{ss}),  with associated functions $f_{ij}(z,z_{1},y,y_1)$, such that $f_{11}=\eta\in\R$ if, and only if, it one of the following two cases occur
\begin{description}
\item [{(i)}]
\beq\label{eq:thIIIBFG}
\left\lbrace\begin{array}{l}
z_{1,t}=\dfrac{1}{W}\left[  -a(b z_1+ay_1)\phi''(\xi)+\eta  (\mu h_{y_1}-\delta\lambda g_{y_1})\phi'(\xi)- \frac{1}{2}(h^2-\delta g^2)_{y_1}\phi (\xi) \right],\\[8pt]
y_{1,t}=\dfrac{1}{W}\left[ b(b z_1+ay_1)\phi''(\xi)-\eta(\mu h_{z_1}-\delta\lambda g_{z_1})\phi'(\xi)+\frac{1}{2}(h^2-\delta g^2)_{z_1}\phi (\xi) \right],
\end{array}\right.
\eeq
where, $\delta=1$ (resp. $-1$), $a,b,\mu,\lambda\in\R$,  $g(z_1,y_1),\, h(z_1,y_1)$  are smooth functions, and $\phi(\xi)$, $\xi=a y+b z$, is a smooth strictly monotone function of $\xi$,  such that 
\beq
& a^2+b^2\neq 0,\qquad \mu^2+\lambda^2\neq 0,\qquad W=g_{z_1}h_{y_1}-g_{y_1}h_{z_1}\neq 0,\nonumber \\
&\mu g-\lambda h= ay_1+b z_1,\qquad g\neq \lambda \eta \frac{\phi'(\xi)}{\phi(\xi)}.
\label{eq:condThIII1}
\eeq
In this case, the associated functions are
\be\label{eq:fijS1f11Caso1}
\begin{array}{ll}
f_{11}=\eta, \qquad & f_{12}=\phi(\xi).\\
f_{21}=g,\qquad & f_{22}=\lambda \phi'(\xi).\\
f_{31}=-h,\qquad & f_{32}=-\mu \phi'(\xi).
\end{array}
\ee
\item [{(i)}]
\beq\label{eq:thIIICFG}
\left\lbrace\begin{array}{ll}
z_{1,t}=-\dfrac{z_1}{\gamma}(p_{zy}+\tau p)-\frac{y_1}{\gamma}(p_{yy}+\beta p)+\frac{\eta}{\gamma^2}(\beta p_{z}-\tau p_{y}),\\
y_{1,t}=\dfrac{z_1}{\gamma}(p_{zz}+\alpha p)+\frac{y_1}{\gamma}(p_{zy}+\tau p)-\frac{\eta}{\gamma^2}(\tau p_{z}-\alpha p_{y}),
\end{array}\right.
\eeq
where
\beq\nonumber
&\gamma:=a_1b_3-b_1a_3\neq 0,\quad \alpha=a_3^2-\delta a_1^2,\quad  \beta=b_3^2-\delta b_1^2,\quad \tau= a_3b_3-\delta a_1b_1,\label{eq:condThIII2}
\eeq
with, $a_1,b_1,a_3,b_3\in\R$, $\delta=1$ (resp. $-1$) and  $\; p(z,y)$ is a smooth function, such that $p_z$ and  $p_y$ are not proportinal. 
In this case, the associated functions are
\be\label{eq:fijS1f11Caso2}
\begin{array}{ll}
f_{11}=\eta, \qquad & f_{12}=p,\\
f_{21}=a_{1}z_1+b_{1}y_1, \qquad & f_{22}=\frac{1}{\gamma}(b_1 p_z -a_1 p_y),\\
f_{31}=-a_{3}z_1-b_{3}y_1, \qquad & f_{32}=-\frac{1}{\gamma}(b_3 p_z -a_3 p_y).
\end{array}
\ee
\end{description}
\end{thm}

In Section \ref{further_ex}, as a consequence of our main results, we will obtain  several corollaries 
% \ref{prop:exp}, \ref{prop:Ccincoparametros} and %\ref{prop:classe5paramentrosCDIS}, 
that  will provide special classes of systems of differential equations contained in these classifications theorems. Most of the examples of this paper are produced by using these corollaries.

\section{Proof of the mains results\label{sec:Proofs}}

The proof of the characterization result, Theorem \ref{Lemma:S10}, is quite similar to the proof of Lemma 1 in \cite{ding}, and therefore it can be omitted here.
We will only present the proof of Theorem  \ref{th:II}, since  the proofs of Theorems \ref{th:I} and \ref{th:II} are also quite similar.  Moreover, Theorem \ref{th:III} can be obtained directly from Theorem \ref{th:II}, since it follows from the fact that the structure equations (\ref{eq:SE}) are invariant under the transformation,
\beq\label{prop:f11}
\begin{array}{ll}
f_{11}\rightarrow f_{21},\qquad &f_{12}\rightarrow f_{22},\nonumber\\
f_{21}\rightarrow f_{11},\qquad &f_{22}\rightarrow f_{12},\nonumber\\
f_{31}\rightarrow -f_{31},\qquad &f_{32}\rightarrow -f_{32}, \nonumber
\end{array}
\eeq
which allows one to consider the hypothesis $f_{11}=\eta$ instead of $f_{21}=\eta$. We point out that, although the systems describing \textbf{pss} and \textbf{ss} with associated functions satisfying $f_{21}=\eta$ or $f_{11}=\eta$ are the same,  we notice that they may have different associated linear problems.

In order to prove Theorem \ref{th:II}, we need a technical lemma, which classifies the local solutions of the following equation
\be
\psi_{0,z}z_1+\psi_{0,y}y_1-\s \rho_1\psi_2+\s \rho_2\psi_1=0, \qquad \s=\pm 1,\label{lema2eq1}
\ee
for functions $\psi_{i}(z,y)$, $i=0,1,2$, depending on $(z,y)$ and functions $\rho_{j}((z_1,y_1)$, $j=1,2$, depending on $(z_1,y_1)$, , satisfying the following generic condition,
\be
\rho_{1,z_1}\rho_{2,y_1}-\rho_{1,y_1}\rho_{2,z_1}\neq 0.\label{lema2eq2}
\ee

\begin{lemma}\label{lema2}
Smooth functions $\psi_{i}(z,y)$, $i=0,1,2$, and $\rho_j(z_1,y_1)$, $j=1,2$, defined on an open connected subset of $\R^4$ satisfy (\ref{lema2eq1}) and the generic condition (\ref{lema2eq2}) if, and only if, one the following holds
%Let $\psi_{i}(z,y)$, $i=0,1,2$, and $\rho_j(z_1,y_1)$, %$j=1,2$, be smooth functions defined on an open connected %subset of $\R^4$ and $\s=\pm 1$. These functions satisfy %(\ref{lema2eq1}), with the generic condition %(\ref{lema2eq2}) if, and only if, one the following holds
\begin{description}
\item [(I)] $\psi_0=c\in\R$ is constant, $\psi_1\equiv\psi_2\equiv 0$  and $\rho_j$ are  arbitrary functions satisfying (\ref{lema2eq2});  %$\rho_{1,z_1}\rho_{2,y_1}-\rho_{1,y_1}\rho_{2,z_1}\neq 0$;
\item [(II)] $\psi_0=\phi(\xi)$ is a smooth stricly monotone function of , $\xi=ay+bz$,  where $a,b\in \R$ are such that  $a^2+b^2\neq  0$;  $\psi_1=\lambda \phi'(\xi)$, $\psi_2=\mu \phi'(\xi)$, where $\lambda$ and $\mu$ are constants, satisfying $\lambda^2+\mu^2\neq 0$, $\rho_1=g$, $\rho_2=h$ where $g,h$ are functions of $(z_1,y_1)$ satisfying $g_{z_1}h{y_1}-g{y_1}h{z_1}\neq 0$ and $\mu g-\lambda h=\s a y_1+\s b z_1$;

\item [(III)] $\rho_i=a_i z_1+b_i y_1$, $i=1,2$ where $a_i,b_i$, $i=1,2$, are constants such that $a_1b_2-b_1a_2\neq 0$, $\psi_0=p$, for $p$ a smooth function of $(z,y)$ such that $p_z$ and $p_y$ are not proportional and $\displaystyle\psi_i=\frac{\s}{a_1b_2-b_1a_2}(b_ip_z-a_ip_y)$, $i=1,2$.
\end{description}
\end{lemma}
\begin{proof}
Let $\psi_{i}(z,y)$, and $\rho_j(z_1,y_1)$ be smooth functions, defined on a open connected subset of $\R^4$,  that  satisfy (\ref{lema2eq1}) and (\ref{lema2eq2}).
Taking the derivative of (\ref{lema2eq1}) with respect to $z_1$ e $y_1$, we have
\be\label{psi0zy}
\left(\begin{array}{c}
\psi_{0,z}\\
\psi_{0,y}
\end{array}\right) =\s
\left(\begin{array}{cc}
-\rho_{2,z_1}&\rho_{1,z_1}\\
-\rho_{2,y_1}&\rho_{1,y_1}
\end{array}\right)
\left(\begin{array}{c}
\psi_{1}\\
\psi_{2}
\end{array}\right), 
\ee
where $\s=\pm 1$. Since (\ref{lema2eq2}) holds, equation (\ref{psi0zy}) is equivalent to
\be\label{psi12}
\left(\begin{array}{c}
\psi_{1}\\
\psi_{2}
\end{array}\right)=\displaystyle\frac{\s}{\rho_{1,z_1}\rho_{2,y_1}-\rho_{1,y_1}\rho_{2,z_1}}
\left(\begin{array}{cc}
\rho_{1,y_1}&-\rho_{1,z_1}\\
\rho_{2,y_1}&-\rho_{2,z_1}
\end{array}\right)
\left(\begin{array}{c}
\psi_{0,z}\\
\psi_{0,y}
\end{array}\right).
\ee 

Now we consider three cases. Firstly, suppose that $\psi_{1}$ and $\psi_{2}$ vanish identically, since $\psi_0$ is defined on a connected set, equation (\ref{psi0zy}) implies that $\psi_0=c\in\R$.  In this case equation (\ref{lema2eq1}) is verified for every $\rho_1$ and $\rho_2$ satisfying (\ref{lema2eq2}). Hence  \textbf{(I)} holds.

Without loss of generality we can assume $\psi_1^2+\psi_2^2\neq 0$ on an open set. Suppose that $\psi_{1}$ and $\psi_{2}$ are linearly dependent. Then  (\ref{psi0zy}) implies that $\psi_{0,z}$ and $\psi_{0,y}$ are linearly dependent as well, i.e., there exists a pair of constants $A,B\in \R$, such that $A^2+B^2\neq 0$ and $A \psi_{0,z}+B\psi_{0,y}=0$. This implies that $\psi_0$ is a solution of a first order partial differential equation, whose  solution is given by
\be
\psi_{0}=\phi (\xi), \qquad \xi =ay+bz,
\ee
 where, $\phi$ is a real smooth function of  $\xi$, where $a=-A$ and $b=B$. Thus, equation (\ref{psi12}) implies that
\beq
\psi_{1}=\displaystyle\frac{\s\phi'(\xi)}{\rho_{1,z_1}\rho_{2,y_1}-\rho_{1,y_1}\rho_{2,z_1}}\left(b\rho_{1,y_1}-a\rho_{1,z_1}\right),\\
\psi_{2}=\frac{\s\phi'(\xi)}{\rho_{1,z_1}\rho_{2,y_1}-\rho_{1,y_1}\rho_{2,z_1}}\left(b\rho_{2,y_1}-a\rho_{2,z_1}\right).
\eeq
Observe that  $\phi'(\xi)\neq 0$, otherwise $\psi_1^2+\psi_2^2=0$  which is a contradiction. Hence, we can divide both sides of the above equations by $\phi'(\xi)$ and  by separation of variables, we conclude that there exists s $\lambda,\mu\in\R$ such that,
\beq
&\psi_{1}=\lambda\phi'(\xi),\label{eq:lema2l}\\
&\psi_{2}=\mu \phi'(\xi),\label{eq:lema2mu}\\
&b\rho_{1,y_1}-a\rho_{1,z_1}-\lambda \s (\rho_{1,z_1}\rho_{2,y_1}-\rho_{1,y_1}\rho_{2,z_1})=0,\label{eq:lema2lW}\\
&b\rho_{2,y_1}-a\rho_{2,z_1}-\mu \s (\rho_{1,z_1}\rho_{2,y_1}-\rho_{1,y_1}\rho_{2,z_1})=0\label{eq:lema2muW}.
\eeq
Equations (\ref{eq:lema2l}) and (\ref{eq:lema2mu}) imply  that $\lambda^2+\mu^2\neq 0$, otherwise $\psi_{1}^2+\psi_{2}^2= 0$ which is a contradiction. Moreover, from (\ref{eq:lema2lW}) and (\ref{eq:lema2muW}) we have
\be\nonumber
\left(\begin{array}{c}
\s a\\
\s b
\end{array}\right)=
\left(\begin{array}{c}
-\lambda \rho_{2,y_1}+\mu \rho_{1,y_1}\\
-\lambda \rho_{2,z_1}+\mu \rho_{1,z_1}
\end{array}\right).
\ee
This implies that $-\lambda\rho_{2}+\mu \rho_{1}$ must depend linearly on $z_1$ and $y_1$ as follows 
\be\nonumber
-\lambda \rho_{2}+\mu \rho_{1}=\s a y_1+ \s b z_1 +c,
\ee
for some constant $c\in\R$. Thus, equation (\ref{lema2eq1}) is equivalent to $c \phi'(\xi)=0$.  Since $\phi'(\xi)\neq 0$ we have $c=0$. Hence, case \textbf{(II)} follows by considering $\rho_1=g$ and $\rho_2=h$, where $g$ and $h$ are smooth functions of $(z_1,y_1)$ satisfying $g_{z_1}h_{y_1}-g_{y_1}h_{z_1}\neq 0$ and $\mu g-\lambda h=\s a y_1+\s b z_1$.

Finally, suppose that $\psi_{1}$ and $\psi_{2}$ are linearly independent. Taking the derivatives of equation (\ref{psi0zy}) with respect to $z_1,y_1$, we conclude that the second derivatives of $\rho_1$ and $\rho_2$ vanish. Therefore, there exist $a_i,b_i,c_i\in\R$, $i=1,2$, such that,
\beq
\rho_{1}=a_1 z_1+b_1 y_1+c_1,\qquad 
\rho_{2}=a_2 z_1+b_2 y_1+c_2.\nonumber
\eeq
Substituting these expressions of $\rho_{1}$ and $\rho_{2}$ into (\ref{psi0zy}), we have
\[
\left(\begin{array}{c}
\psi_{0,z}\\
\psi_{0,y}
\end{array}\right) =
\s\left(\begin{array}{cc}
-a_2&a_1\\
-b_2&b_1
\end{array}\right)
\left(\begin{array}{c}
\psi_{1}\\
\psi_{2}
\end{array}\right) .
\]
On other hand, equation (\ref{lema2eq1}) is equivalent to
$\,
-c_1 \psi_{2}+c_2 \psi_{1}=0.
\, $
Since $\psi_{1}$ and $\psi_{2}$ are linearly independent, it follows that $c_1=c_2=0$. Equation  (\ref{psi12}) implies that  
\[
\psi_{i}=\frac{\s}{a_1b_3-b_1a_3}(b_i \psi_{0,z}-a_i \psi_{0,y}),
\]
for  $i=1,2$. Moreover, denoting $\psi_0=p(z,y)$ we have that  $p_{z}$ and $p_{y}$ are linearly independent, which shows that \textbf{(III)} holds. 

The converse of the lemma is a straightforward computation.
\end{proof}

\subsection*{Proof of Theorem \ref{th:II}}
\begin{proof} Let $f_{ij}(z,z_1,y,y_1)$ be smooth functions defined on an open connected subset of $\R^4$. It follows from   Theorem \ref{Lemma:S10} that equations (\ref{eq:lemma1-1})-(\ref{eq:lemma1-6}) hold. The hypothesis  $f_{21}=\eta\in\R$ implies that (\ref{eq:lemma1-2}) is equivalent to
\be\label{eq:thII-1}
W:=f_{11,z_1}f_{31,y_1}-f_{11,y_1}f_{31,z_1}\neq 0.
\ee
Hence, from (\ref{eq:lemma1-3}) and (\ref{eq:lemma1-5}) we have
\be\label{eq:FGproofII}
\left(
\begin{array}{c}
F\\
G
\end{array}
\right)=\frac{1}{W}
\left(
\begin{array}{cc}
f_{31,y_1}&-f_{11,y_1}\\
-f_{31,z_1}&f_{11,z_1}
\end{array}
\right)\left(
\begin{array}{c}
f_{12,z}z_1+f_{12,y}y_1-f_{31}f_{22}+f_{32}\eta\\
f_{32,z}z_1+f_{32,y}y_1-\delta f_{11}f_{22}+\delta f_{12}\eta
\end{array}
\right) .
\ee
On the other hand (\ref{eq:lemma1-4}) is equivalent to
\be\label{eq:IIL4}
f_{22,z}z_1+f_{22,y}y_1-f_{11}f_{32}+f_{12}f_{31}=0.
\ee

Considering $\psi_0=f_{22}$, $\psi_1=f_{12}$, $\psi_2=f_{32}$, $\rho_1=f_{11}$, $\rho_2=f_{31}$ and $\s=1$, equations (\ref{eq:IIL4}) and (\ref{eq:thII-1}) reduce to (\ref{lema2eq1}) and (\ref{lema2eq2}). Thus, according to Lemma \ref{lema2} one has to examine three cases.  We will show that case \textbf{(I)} contradicts  equation (\ref{eq:irr}), whereas cases \textbf{(II)} and \textbf{(III)} lead to  \textbf{(i)} and \textbf{(ii)}.

Firstly, suppose that the solution is of type \textbf{(I)} of Lemma \ref{lema2}, i.e., $f_{12}=f_{32}=0$, $f_{22}=\lambda\in\R$ and $f_{11},f_{31}$, are smooth functions satisfying $f_{11,z_1}f_{31,y_1}-f_{11,y_1}f_{31,z_1}\neq 0$. In this case,  equation (\ref{eq:FGproofII}) is equivalent to
\be\nonumber
\left(
\begin{array}{c}
F\\
G
\end{array}
\right)=\frac{1}{f_{11,z_1}f_{31,y_1}-f_{11,y_1}f_{31,z_1}}
\left(
\begin{array}{cc}
f_{31,y_1}&-f_{11,y_1}\\
-f_{31,z_1}&f_{11,z_1}
\end{array}
\right)\left(
\begin{array}{c}
-f_{31}\lambda\\
-\delta f_{11}\lambda
\end{array}
\right) .
\ee
Hence,  $F$ and $G$ are functions only of $z_1$ and $y_1$,  thus it does not satisfy equation (\ref{eq:irr}), which is a contradiction.

 Case \textbf{(i)}. Suppose that the solution of (\ref{eq:IIL4}), with condition (\ref{eq:thII-1}), is of  type \textbf{(II)} of Lemma \ref{lema2}. Hence, $f_{22}=\phi (\xi)$,  is a real strictly monotone smooth function of $\xi=ay+bz$, where $a,b\in \R$, with  $a^2+b^2\neq 0$, $f_{12}=\lambda \phi'(\xi)$, $f_{32}=\mu\phi'(\xi)$, where $\lambda,\mu\in\R$, with  $\lambda^2+\mu^2\neq 0$, $f_{11}=g$, $f_{31}=h$, where $g,h$ are smooth functions of $(z_1,y_1)$ satisfying, $g_{z_1}h_{y_1}-g_{y_1}h_{z_1}\neq 0$ and $\mu g-\lambda h=ay_1+bz_1$. 

In this case, equation (\ref{eq:lemma1-6}) is equivalent to 
 \[
 g\neq \lambda\eta\frac{\phi'(\xi)}{\phi(\xi)}.
 \]
Thus, we conclude that the relations (\ref{eq:condThII1}) are necessary. Moreover, equation (\ref{eq:FGproofII}) is equivalent to the system (\ref{eq:thIIBFG}). We claim that this system satisfies (\ref{eq:irr}). In fact,  otherwise, suppose  that $F$ and $G$ do not depend on $z$ and $y$. From  equation (\ref{eq:FGproofII}), we have 
\be\label{eq:irredth1B}
\left(
\begin{array}{cc}
g_{z_1}&g_{y_1}\\
h_{z_1}&h_{y_1}
\end{array}
\right)\left(
\begin{array}{c}
F\\
G
\end{array}
\right)=\left(
\begin{array}{c}
\lambda b \phi''(\xi) z_1+a\lambda\phi''(\xi)y_1-h\phi(\xi)+\eta\mu\phi'(\xi)\\
\mu b \phi''(\xi) z_1+\mu a \phi''(\xi) y_1-\delta g\phi(\xi)+\delta \eta\lambda\phi'(\xi)
\end{array}
\right) .
\ee

Taking the derivative of first line, with respect to $z$ and $y$, we have,
\[
b\left[b\lambda\phi'''(\xi)z_1+a\lambda\phi'''(\xi)y_1-h\phi'(\xi)+\eta\mu\phi''(\xi)\right]=0,\\
a\left[b\lambda\phi'''(\xi)z_1+a\lambda\phi'''(\xi)y_1-h\phi'(\xi)+\eta\mu\phi''(\xi)\right]=0.
\]
Since $a^2+b^2\neq 0$, we conclude that 
\be\nonumber
\lambda(bz_1+ay_1)\phi'''(\xi)+\eta\mu\phi''(\xi)-\phi'(\xi) h=0.
\ee
Similarly, taking the derivatives of  the second line of (\ref{eq:irredth1B}), we obtain 
\be\nonumber
\mu(b z_1+ay_1)\phi'''(\xi)+\delta\eta\lambda\phi''(\xi)-\delta g\phi'(\xi)=0.
\ee
The last two equations imply  that
$$
\eta(\mu^2-\delta\lambda^2)\phi''(\xi)-\mu\phi'(\xi) h+\delta \lambda g\phi'(\xi)=0.
$$
Differentiating this equation  with respect to $z_1$ and $y_1$ we get,
\be\nonumber
\left(\begin{array}{cc}
h_{z_1}&g_{z_1}\\
h_{y_1}&g_{y_1}
\end{array}\right)\left(\begin{array}{c}
-\mu\phi'(\xi)\\
\delta \lambda\phi'(\xi)
\end{array}\right)=
\left(\begin{array}{c}
0\\
0
\end{array}\right).
\ee
Since, $h_{z_1}g_{y_1}-g_{z_1}h_{y_1}\neq 0$, we conclude  that $\mu\phi'(\xi)=\lambda\phi'(\xi)=0$, which  contradicts the fact that $\phi'(\xi)\neq 0$ and $\lambda^2+\mu^2\neq 0$. This proves that if the functions $f_{ij}$ satisfying  
(\ref{eq:IIL4}) and (\ref{eq:thII-1})
 are of type  \textbf{(II)} of Lemma \ref{lema2}, then  case \textbf{(i)} of  Theorem \ref{th:II} necessarily holds. In order to show the converse, define $f_{ij}$ as in  (\ref{eq:fijS1f21Caso1}), then  a straightforward computation shows that equations (\ref{eq:lemma1-1})-(\ref{eq:lemma1-6}) are satisfied.

Case \textbf{(ii)}. Suppose that the solution of (\ref{eq:IIL4}), with condition (\ref{eq:thII-1}) is of type \textbf{(III)} of Lemma \ref{lema2}. Then, there exist, $a_1,a_3,b_1,b_3\in \R$, such that $a_1b_3-b_1a_3\neq 0$, $f_{1i}=a_i z_i+b_iy_i$, for $i=1,3$, $f_{22}=p$, where $p$ is a smooth function of $(z,y)$, such that, $p_z$ and $p_y$ are linearly independent and $f_{i2}=\frac{1}{a_1b_3-b_1a_3}(b_i p_z-a_i p_y)$, for $i=1,3$.

In this case, equation (\ref{eq:lemma1-6}) is equivalent to
\[
a_1[p (a_1b_3-b_1a_3) z_1+\eta p_y]+b_1[p (a_1b_3-b_1a_3) y_1-\eta p_z]\neq 0,
\]
which holds , up to a set of measure zero. In fact,  otherwise we would have  $(a_1b_3-b_1a_3) a_1 p=(a_1b_3-b_1a_3) b_1 p=0$ which is a contradiction, since $a_1b_3-b_1a_3\neq 0$, $p\neq 0$ and $a_1,b_1$ do not vanish  simultaneously. Moreover, equation (\ref{eq:FGproofII}) is equivalent to the system (\ref{eq:thIICFG}), where,
\be\nonumber
\gamma= a_1b_3-b_1a_3,\quad \alpha=a_3^2-\delta a_1^2,\quad \beta= b_3^2-\delta b_1^2\quad\mbox{e}\quad\tau=a_3b_3-\delta a_1b_1 .
\ee

We claim that system (\ref{eq:thIICFG}) satisfies equation (\ref{eq:irr}). Indeed, 
\beq
F&=&-\frac{z_1}{\gamma}(p_{zy}+\tau p)-\frac{y_1}{\gamma}(p_{yy}+\beta p)+\frac{\eta}{\gamma^2}(\beta p_{z}-\tau p_{y}),\nonumber\\
G&=&\frac{z_1}{\gamma}(p_{zz}+\alpha p)+\frac{y_1}{\gamma}(p_{zy}+\tau p)-\frac{\eta}{\gamma^2}(\tau p_{z}-\alpha p_{y}).\nonumber
\eeq
satisfy a stronger condition, namely, $F^2_z+G^2_y\neq 0$. 
Otherwise,  differentiating $F$ with respect to $z$ and $G$ with respect to $y$, we would have
\beq
p_{zzy}+\tau p_z=0,\qquad p_{zyy}+\beta p_z=0,\nonumber\\
p_{zzy}+\alpha p_y=0,\qquad p_{zyy}+\tau p_y=0,\nonumber
\eeq
which imply 
\be\nonumber
\left(\begin{array}{cc}
\alpha &\tau\\
\tau &\beta 
\end{array}\right)\left(\begin{array}{c}
-p_y\\
p_z
\end{array}\right)=
\left(\begin{array}{c}
0\\
0
\end{array}\right).
\ee
Since $\alpha\beta-\tau^2=-\delta\gamma^2\neq 0$, we have $p_z=p_y=0$, which is a contradiction. Hence we conclude that 
if the functions $f_{ij}$ satisfying  
(\ref{eq:IIL4}) and (\ref{eq:thII-1})
 are of type  \textbf{(III)} of Lemma \ref{lema2}, then case \textbf{(ii)} of  Theorem \ref{th:II}  necessarily holds. 
In order to prove the converse, we define $f_{ij}$, as in (\ref{eq:fijS1f21Caso2}) and we show directly that (\ref{eq:lemma1-1})-(\ref{eq:lemma1-6}) are satisfied.

\end{proof}

\section{Applications \label{further_ex}}
In this section, we obtain several corollaries of Theorems \ref{th:I} and \ref{th:II}, which provide  new multi-parameter families of systems describing \textbf{pss} or \textbf{ss}. Particular choices of the parameters give equations such as  the Pohlmeyer-Lund-Regge type of equation (\ref{eq:mPLR1}) and the Konno-Oono coupled dispersionless system. 

For the applications considered in this section, we return to the original 
notation for systems of hyperbolic differential equations \eqref{eq:S} for functions $u(x,t)$ and $v(x,t)$. Therefore, in Theorems \ref{th:I} and \ref{th:II}, we replace  $\,z,\, z_1,\, y,\, y_1\,$ by  $\,u,\, u_x,\, v,\,v_x\,$
respectively.

All the systems, obtained as applications of the main theorems, have 
one-parameter family of associated functions $f_{ij}$, implying that they have a one-parameter family of associated liner problems (\ref{eq:ProblemaLinearlocal}) or (\ref{eq:ProblemaLinearlocal3x3}).

\medskip{}
% \subsection{Família de sistemas a sete parâmetros generalizando a equação (vetorial) de Pohlmeyer-Lund-Regge modificada. }
% Nesta subseção, utilizando o Teorema \ref{th:II} Caso $2)$, apresentaremos uma família a sete parâmetros de sistemas que descrevem superfícies pseudo-esféricas ou esféricas que generaliza a equação (vetorial) de Pohlmeyer-Lund-Regge modificada (\ref{eq:mPLR1}).

\begin{cor}\label{cor:7mprlpseudoesferica} 
The 7-parameter system of differential equations
 \be
\label{eq:7mprlpeseudoeferica}
\left\lbrace\begin{array}{l}
u_{xt}= ab \sin\theta\,[u_x\psi(u,v)-k_2]-b^2\cos\theta\,[v_x\psi(u,v)+k_1]+ k_0 u,\\
v_{xt}=- a^2 \cos\theta\,[u_x\psi(u,v)-k_2]-ab \sin\theta\,[v_x\psi(u,v)+k_1]+k_0 v,
\end{array}\right. 
\ee
 where $k_0,k_1,k_2,k_3,a,b,\theta\in\R$, $k_0ab\neq 0$ and $\psi$ is given by,
\be
\psi(u,v)=k_0\left(\displaystyle\frac{\cos\theta}{2ab}(a^2u^2-b^2v^2)+uv\sin\theta \right)-ab(k_1u+k_2v+k_3),\nonumber
\ee
describes \textbf{pss}, with associated functions 
\be\nonumber
\begin{array}{ll}
f_{11}=\eta\left[a \sin\left(\frac{\theta}{2}\right)u_x-b \cos\left(\frac{\theta}{2}\right)v_x\right] ,&\qquad f_{12}=\frac{1}{\eta}\left[b\cos\left(\frac{\theta}{2}\right)\psi_u(u,v)+a\sin\left(\frac{\theta}{2}\right)\psi_v(u,v)\right],\\
f_{21}=\eta^2 ,&\qquad f_{22}=\frac{1}{\eta^2}k_0-ab\psi(u,v),\\
f_{31}=\eta\left[a \cos\left(\frac{\theta}{2}\right)u_x+ b \sin\left(\frac{\theta}{2}\right)v_x\right],& \qquad f_{32}=\frac{1}{\eta}\left[- b \sin\left(\frac{\theta}{2}\right)\psi_u(u,v)+a\cos\left(\frac{\theta}{2}\right)\psi_v(u,v)\right],
\end{array}
\ee
where $\eta\neq 0$.
\end{cor}

\begin{proof} 
Corollary \ref{cor:7mprlpseudoesferica} is an immediate consequence of  Theorem \ref{th:II}, \textbf{(ii)}, by setting $\delta=1$, $\eta\rightarrow \eta^2$, $p(u,v)=\frac{1}{\eta^2}k_0-ab\psi(u,v)$ and
\[
a_1=\eta a \sin\left(\frac{\theta}{2}\right),\quad a_3=\eta a\cos\left(\frac{\theta}{2}\right),\quad 
b_1=-\eta b \cos\left(\frac{\theta}{2}\right) ,\quad b_3=\eta b \sin\left(\frac{\theta}{2}\right).
\]
\end{proof}

\begin{cor}\label{cor:7mprlesferica} 
The 7-parameter system of differential equations
\be\label{eq:7mprleferica}
\left\lbrace\begin{array}{l}
u_{xt}= ab \sinh\theta(u_x\psi(u,v)-k_2)-b^2\cosh\theta(v_x\psi(u,v)+k_1)+ k_0 u,\\
v_{xt}=- a^2 \cosh\theta(u_x\psi(u,v)-k_2)-ab \sinh\theta(v_x\psi(u,v)+k_1)+k_0 v,
\end{array}\right. 
\ee
 where $k_0,k_1,k_2,k_3,a,b,\theta\in\R$, $k_0ab\neq 0$ and $\psi$ is given by,
\be
\psi(u,v)=k_0\left(\displaystyle\frac{\cosh\theta}{2ab}(a^2u^2-b^2v^2)+\sinh\theta uv\right)-ab(k_1u+k_2v+k_3),\nonumber
\ee
describes \textbf{ss}, with  associated functions 
\be\nonumber
\begin{array}{ll}
f_{11}=\eta\left[a \sinh\left(\frac{\theta}{2}\right) u_x-b \cosh\left(\frac{\theta}{2}\right)v_x\right] ,&\qquad f_{12}=\frac{1}{\eta}\left[ b\cosh\left(\frac{\theta}{2}\right)\psi_u(u,v)+a\sinh\left(\frac{\theta}{2}\right)\psi_v(u,v)\right],\\
f_{21}=\eta^2 ,&\qquad f_{22}=\frac{1}{\eta^2}k_0-ab\psi(u,v),\\
f_{31}=\eta\left[a \cosh\left(\frac{\theta}{2}\right)u_x- b \sinh\left(\frac{\theta}{2}\right)v_x\right],& \qquad f_{32}=\frac{1}{\eta}\left[ b \sinh\left(\frac{\theta}{2}\right)\psi_u(u,v)+a\cosh\left(\frac{\theta}{2}\right)\psi_v(u,v)\right],
\end{array}
\ee
where $\eta\neq 0$.
\end{cor}
\begin{proof}
Similarly to the previous case, we obtain Corollary \ref{cor:7mprlesferica} as an immediate consequence of  Theorem \ref{th:II}, \textbf{(ii)}, by setting $\delta=-1$, $\eta\rightarrow \eta^2$, $p(u,v)=\frac{1}{\eta^2}k_0-ab\psi(u,v)$ and
\[
a_1=\eta a \sinh\left(\frac{\theta}{2}\right),\quad a_3=\eta a\cosh\left(\frac{\theta}{2}\right),\quad 
b_1=-\eta b \cosh\left(\frac{\theta}{2}\right) ,\quad b_3=-\eta b \sinh\left(\frac{\theta}{2}\right).
\]
\end{proof}

\remark{\rm
i) The Pohlmeyer-Lund-Regge type system (\ref{eq:mPLR1}) belongs to the 7-parameter family (\ref{cor:7mprlpseudoesferica}) by choosing  $k_0=-1$, $k_1=k_2=k_3=0$, $a=\sqrt{2}$, $b=-\sqrt{2}$ and $\theta=\frac{\pi}{2}$.  
Hence,  (\ref{eq:7mprlpeseudoeferica}) is a system of hyperbolic differential equations describing \textbf{pss} that generalizes (\ref{eq:mPLR1}).

ii) By considering $k_0=1$, $k_1=k_2=0$, $k_3=\frac{c}{4}$, $a=-\sqrt{2}$, $b=\sqrt{2}$ and $\theta=\pi$ in (\ref{eq:7mprlpeseudoeferica}), the resulting differential system is the example given by (\ref{eq:7mplrex1}).

ii) By choosing $k_0=-1$, $k_1=k_2=0$, $k_3=\frac{c}{4}$ e $a=-\sqrt{2}$ $b=\sqrt{2}$ and $\theta=0$ in (\ref{eq:7mprleferica}), the resulting differential system is the example given by (\ref{eq:7mplrex2}).
}

\begin{cor}\label{prop:exp}  For any strictly monotone smooth function  $\psi(v_x)$, the following system of differential equations
\beq\label{eq:Classeaquatroparametros}
\left\{
\begin{array}{l}
u_{xt}=k_1e^{k_0u}\psi(v_x),\\
v_{xt}=k_1(\delta k_2^{-2}-k_0^2)e^{k_0u}\displaystyle\frac{u_x}{\psi'(v_x)},
\end{array}
\right.
\eeq
where  $k_0,k_1,k_2\in \R\setminus \{0\}$,  
describes \textbf{pss} (resp. \textbf{ss}) for $\delta=1$ (resp. $\delta=-1$),  with associated functions 
\be\nonumber
\begin{array}{lll}
f_{11}=k_2^{-1}u_x, & f_{21}=\eta, & f_{31}=\eta k_0k_2+\psi(v_x),\\
f_{12}=0, & f_{22}=-k_1k_2^{-1}e^{k_0u}, & f_{32}=-k_0k_1e^{k_0u},
\end{array}
\ee
where $\eta\in \R$.
\end{cor}
\begin{proof}
This corollary follows from Theorem \ref{th:II}, \textbf{(i)}, by choosing $a=0$, $b\neq 0$, $\lambda=0$, $\mu=k_2b$, $h=\eta k_0k_2+\psi(v_x)$, $g=k_2^{-1}u_x$, $\phi(\xi)=-k_1k_2^{-1}e^{\frac{k_0}{b}\xi}$, $\xi=bu$ and $\delta=1$ (resp. $-1$).
\end{proof}
\medskip{}

\remark{\rm Some of the examples presented in Section \ref{section-mainresults} were obtained from Corollary \ref{prop:exp}. In fact, 
by choosing $\delta=-1$, $k_0=k_1=k_2=1$ and $\psi=av_x+b$, with $a\neq 0$ and $b\in\R$, the system (\ref{eq:Classeaquatroparametros}) reduces to (\ref{eq:exeex1}).  The choices $\psi(v_x)=e^{-v_x}$, $k_0=1$, $k_1=2$, $k_2=-\sqrt{2}$ and $\delta=1$ provide the system (\ref{eq:exeex2}) and by choosing $\psi(v_x)=v_x$, $k_0=1$, $k_1=-2$, $k_2=-\sqrt{2}$, $\delta=1$ we obtain (\ref{eq:exeex3}).
}

\begin{cor}\label{prop:Ccincoparametros} 
Consider the family of systems 
\beq\label{eq:Ccincoparametros}
\left\{
\begin{array}{l}
u_{xt}=\displaystyle\frac{k_0}{k_2q_{v_x}-k_1q_{u_x}}\left[q_{v_x}+(k_1v+k_2u+k_3)\left(\delta k_1 (k_1v_x+k_2u_x)-qq_{v_x}\right)\right],\\
v_{xt}=\displaystyle\frac{-k_0}{k_2q_{v_x}-k_1q_{u_x}}\left[q_{u_x}+(k_1v+k_2u+k_3)\left(\delta k_2 (k_1v_x+k_2u_x)-qq_{u_x}\right)\right].
\end{array}
\right. 
\eeq
where  $k_0,k_1,k_2\in\R\setminus\{0\}$, $ k_3\in \R$  and  $q(u_x,v_x)$ is  an arbitrary  smooth function of $(u_x,v_x)$, such that $k_2q_{v_x}-k_1q_{u_x}\neq 0$.
Then (\ref{eq:Ccincoparametros}) describes \textbf{pss} (resp. \textbf{ss})
when $\delta=1$ (resp. $\delta=-1$), with associated functions 
\be\nonumber
\begin{array}{ll}
f_{11}=\eta(k_1v_x+k_2u_x),&\qquad f_{12}=0,\nonumber\\
f_{21}=\eta^2 ,&\qquad f_{22}=k_0(k_1v+k_2u)+k_0k_3,\nonumber\\
f_{31}=\eta q(u_x,v_x),& \qquad f_{32}=\frac{1}{\eta}k_0.\nonumber
\end{array}
\ee
where $\eta\neq 0$ a real parameter.
\end{cor}
\begin{proof}
This corollary follows from Theorem \ref{th:II}, \textbf{(i)}, by replacing  $\eta\rightarrow \eta^2$ and considering 
\[a=k_0k_1, \qquad  b=k_0k_2, \qquad 
\lambda=0, \qquad  \mu=\frac{1}{\eta}k_0,\]
\[
\phi(\xi)= \xi+k_0k_3, \qquad  \xi=k_0(k_1v+k_2u),\qquad 
h=\eta q(u_x,v_x), \qquad  g= \eta (k_1v_x+k_2u_x).\]
\end{proof}

\remark{\rm By choosing $k_0=a$, $k_1=0$, $k_2=a$, $k_3=b$ and  $q(u_x,v_x)=-\phi(v_x)$,  the system (\ref{eq:Ccincoparametros}) reduces to  Example \ref{eq:exquatroparametros1}.
}

\begin{cor}\label{prop:classe5paramentrosCDIS} 
The family of systems 
\beq\label{eq:Classe5parametrosKODIS}
\left\{
\displaystyle\begin{array}{l}
u_{xt}=\delta\displaystyle\frac{q\,q_{v_x}+k_0^2k_1(k_1v_x+k_2u_x)}{k_2q_{v_x}-k_1q_{u_x}}(k_1v+k_2u+k_3),\\[1.5em]
v_{xt}=-\delta\displaystyle\frac{q\,q_{u_x}+k_0^2k_2(k_1v_x+k_2u_x)}{k_2q_{v_x}-k_1q_{u_x}}(k_1v+k_2u+k_3), 
\end{array}
\right.
\eeq
where $k_0,k_1,k_2 \in \R\setminus\{0\}$, $k_3\in \R$ and  $q(u_x,v_x)$ is an arbitrary smooth function of $(u_x,v_x)$, such that $k_2q_{v_x}-k_1q_{u_x}\neq 0$,  describes \textbf{pss} (resp. \textbf{ss}), when
 $\delta=1$ (resp. $\delta=-1$), with associated functions 
\be\nonumber
\begin{array}{ll}
f_{11}=\frac{\delta }{\nu}k_0 (k_1v_x+k_2u_x),&\qquad f_{12}=0,\\
f_{21}=\,\frac{1}{\nu}q(u_x,v_x),&\qquad f_{22}=\nu,\\
f_{31}=0,& \qquad f_{32}=k_0(k_1v+k_2u+k_3), 
\end{array}
\ee
where $\nu \neq 0$ is a real parameter.
 \end{cor}
\begin{proof}
This corollary is obtained from  Theorem \ref{th:I}, by setting $\eta=0$ and choosing
\[
a=k_1,\qquad b=k_2,\qquad 
\mu =\nu k_0,\qquad \lambda =0,\]
\[
h=\frac{1}{\nu}q(u_x,v_x),\qquad g=\displaystyle\frac{\delta k_0}{\nu}(k_1v_x+k_2u_x),\qquad 
\phi(\xi)=k_0(\xi +k_3),\qquad \xi=k_1v+k_2u.
\]
\end{proof}

\remark{\rm We notice that Konno-Oono coupled dispersionless system (\ref{eq:KOCDS}) can be obtained from (\ref{eq:Classe5parametrosKODIS}) by choosing $k_0=1$, $k_1=2$, $k_2=0$, $q(u_x,v_x)=2u_x$ and $\delta=1$. Hence  (\ref{eq:Classe5parametrosKODIS}) is an infinite  family of systems, involving  3 parameters and one arbitrary function, that generalize the Konno-Oono coupled dispersionless system.

We point out that Example  \ref{eq:cincoparametrosKODISex1}, can also be obtained from (\ref{eq:Classe5parametrosKODIS}) by choosing  $k_0=1$, $k_1=0$, $k_2=1$, $k_3=0$ and $q(u_x,v_x)=u_xv_x$.
}

\end{document}